\documentclass[12pt]{article}
\usepackage{amsmath}
\usepackage{amsfonts}
\usepackage{amssymb}
\usepackage{graphicx}
\usepackage{amsthm}
\usepackage{enumerate}
\usepackage{url,xargs}
\usepackage{hyperref}

\begin{document}

\title{Bounded-excess flows in cubic graphs}

 \author{
 Michael Tarsi
\thanks{Research supported in part by FWF-grant P27615-N25,
headed by Herbert Fleischner}\\
 {\normalsize The Blavatnik School of Computer Science}\\
 {\normalsize Tel-Aviv university, Israel}\\
 {\normalsize tarsi@post.tau.ac.il}\\
 }

%\author
%{Michael Tarsi \footnote{The Blavatnik School of Computer Science,
    %Tel Aviv University, Israel. E-mail: {\tt
    %tarsi@post.tau.ac.il}} \footnote{Partially supported by FWF-grant P27615-N25,
    %headed by Herbert Fleischner. .}}

    %\and Author 2\footnote{Aff)

\maketitle

\newtheorem{theorem}{Theorem}
\newtheorem{lemma}[theorem]{Lemma}
\newtheorem{claim}[theorem]{Claim}
\newtheorem{rrule}[theorem]{Rule}
\newtheorem{definition}[theorem]{Definition}
\newtheorem{notation}[theorem]{Definitions and Notation}
\newtheorem{example}[theorem]{Example}
\newtheorem{corollary}[theorem]{Corollary}
\newtheorem{conjecture}[theorem]{Conjecture}
\newtheorem{observation}[theorem]{Observation}
\newtheorem{proposition}[theorem]{Proposition}

\newcommand{\Po}{\cal{P}^*}
\newcommand{\Pa}{\cal{P}}
\newcommand{\Qo}{\cal{Q}^*}
\newcommand{\Qa}{\cal{Q}}
\def\scc{\mathop{\mathrm{scc}}\nolimits}

\begin{abstract} \noindent
An $(r,\alpha)$-bounded excess flow $((r,\alpha)$-flow) in an
orientation of  a graph $G=(V,E)$ is an assignment $f:E
\rightarrow [1,r-1]$, such that for every vertex $x\in V$,
$|\sum_{e\in E^+(x)}f(e)-\sum_{e\in E^-(x)}f(e)| \leq \alpha$.
 $E^+(x)$, respectively $E^-(x)$, are the sets of edges
directed from, respectively toward $x$. Bounded excess flows
suggest a generalization of Circular nowhere zero flows, which can
be regarded as $(r,0)$-flows.
 We define $(r,\alpha)$ as {\bf Stronger} or equivalent to $(s,\beta)$
If the existence of an $(r,\alpha)$-flow in a cubic graph always
implies the existence of an $(s,\beta)$-flow in the same graph. We
then study the structure of the bounded excess flow strength
poset. Among other results, we define the {\bf Trace} of a point
in the $r$-$\alpha$ plane by
$tr(r,\alpha)=\frac{r-2\alpha}{1-\alpha}$ and prove that among
points with the same trace the stronger is the one with the
smaller $\alpha$ (and larger $r$). e.g. If a cubic graph admits a
$k$-nzf (trace $k$ with $\alpha=0$) then it admits an
$(r,\frac{k-r}{k-2})$-flow for every $r$, $2\leq r \leq k$.
  A significant part of the article is devoted to proving the main result: Every cubic graph admits a
  $(3\frac 12,\frac 12)$-flow, and there exists a graph which does not
  admit any stronger bounded excess flow. Notice that $tr(3\frac 12,\frac
  12)=5$ so it can be considered a step in the direction of the $5$-flow Conjecture.
  Our result is the best possible for all cubic graphs while the
  seemingly stronger
  5-flow Conjecture relates only to bridgeless graphs.
We also show that if the circular flow number of a cubic graph is
strictly less than $5$ then it admits a $(3\frac 13,\frac13)$-flow
(trace 4). We conjecture such a flow to exist in every cubic graph
with a perfect matching, other than the Petersen graph. This
conjecture is a stronger version of the Ban-Linial Conjecture
\cite{banlin}, \cite{emt2}. Our work here strongly rely on the
notion of  {\bf Orientable $k$-weak bisections}, a certain type of
$k$-weak bisections. $k$-weak bisections are defined and studied
in \cite{emt2}.

\end{abstract}

%% \noindent \textit{Keywords: cubic graphs, circular flow number, internal partitions.\\
%%  MSC(2010):}

\section{Introduction}\label{sec:flo}
\subsection{Preliminaries}\label{sec:flo}

We assume familiarity with the theory of {\bf Nowhere-zero flows}
(nzf) (See \cite{cq} for a thorough study) and {\bf Circular
nowhere-zero flows} (cnzf).
\smallskip

\begin{definition}
 Given two real  numbers $r\geq 2$ and $\alpha \geq 0$, an $(r,\alpha)$-{\bf bounded excess flow}, $(r,\alpha)$-{flow} for short, in a directed graph $D=(V,E)$ is
 an assignment $f:E \rightarrow [1,r-1]$, such that $f$ is a {\bf flow} in $D$, with possibly some
  deficiency or excess, which does not exceed  $\alpha$ per vertex. That is,
 for every vertex $x\in V$, $|\sum_{e\in E^+(x)}f(e)-\sum_{e\in
E^-(x)}f(e)| \leq \alpha$, where $E^+(x)$, respectively $E^-(x)$,
is the set of edges directed from, respectively toward $x$ in $D$.
\end{definition}
With that notation an $r$-cnzf
 can be referred to as an $(r,0)$-flow. Notice that while nowhere
 zero flows are restricted to bridgeless graphs, this
 is  not the case for bounded excess flows, where a bridge can
 carry some excess from one of its "sides" to the other.

 We say that an undirected graph $G$ admits an $(r,\alpha)$-flow if there exists an orientation of $G$ which admits
 such a flow.

In this article we study bounded excess flows in {\bf cubic}
graphs. A graph may have parallel edges, but no loops.

\begin{notation} \label{parcol}

(Rather than a stand-alone definition, we post under that title
lists of related definitions and notational remarks).
\begin{itemize}
 \item A {\bf Vertex partition} of a graph $G=(V,E)$ is a partition $\Psi=(V_1,V_2)$ of its vertex set $V$
 into two disjoint subsets, $V=V_1 \cup V_2$,  $V_1 \cap V_2=\emptyset$.
 \item Whenever it comes convenient, we use vertex partition and vertex
 {\bf coloring} $\Psi:V\rightarrow \{1,2\}$ as synonyms, where $\Psi(v)=i\Leftrightarrow v\in V_i\Leftrightarrow$ "the color of $v$ is
 $i$".
\item Once interpreted as a vertex 2-coloring  a vertex partition of a
graph naturally induces a vertex partition of every subgraph, and
conversely, the {\bf union} of vertex partitions of vertex
disjoint subgraphs naturally provides a coloring of their graph
union.
\item A {\bf Bisection} of $G=(V,E)$ is a vertex partition $(V_1,V_2)$ into two subsets of equal size,
 $|V_1|=|V_2|$.
\end{itemize}
Let $(V_1,V_2)$ be a vertex partition of a graph $G=(V,E)$ and let
$A\subseteq V$ be any set of vertices. We then use the following
notation:
\begin{itemize}
\item The set of edges with one
 endvertex in $A$ and the other one in $V \setminus A$, known as the
 \textbf{edge-cut} induced by $A$, is denoted here by $E(A)$. Its
 cardinality $|E(A)|$ is denoted by $d(A)$. When more than one
 graph (or subgraph) is involved we can use $d_G(A)$ to avoid
 ambiguity.
\item Subject to a given orientation of $G$, $E(A)$ is partitioned into $E^+(A)$ and $E^-(A)$, the sets of edges directed from $A$ and into $A$.
We also denote in that case
 $d^+(A)=|E^+(A)|$ and $d^-(A)=|E^-(A)|$.
\item Given a vertex partition $\Psi=(V_1,V_2)$, $\delta_{\Psi}(A)=|V_2 \cap A|-|V_1 \cap A|$
and $\Delta_{\Psi}(A)=|\delta_{\Psi}(A)|$. When there is no
confusion we omit $\Psi$ and use $\delta(A)$ and $\Delta(A)$.
\end{itemize}

 The following terminology is restricted to cubic graphs:
\begin{itemize}
\item An orientation of a cubic graph is {\bf Balanced} if the
outdegree of every vertex is either $1$ or $2$ (namely, there is
no vertex of outdegree $0$ or $3$).
\item We say that a bisection $\Psi=(V_1,V_2)$ of a cubic graph $G$ is {\bf Orientable} if there exists a
 (clearly balanced) orientation
of $G$  where $V_i$, $i=1,2$ are the sets of vertices with
outdegree $i$, or equivalently for every vertex $v$,
$\Psi(v)=d^+(v)$.

\end{itemize}
\end{notation}

The following is a simple instance of a theorem on feasible
orientations of graphs. As we did not find a direct reference to
that specific instance, we leave it as an exercise for the reader
(A hint: Use Hall's  theorem to match vertices $\times$ outdegree
into edges):

\begin{lemma} \label{orientable} A bisection $(V_1,V_2)$ of a
cubic graph  $G=(V,E)$ is orientable if and only if $d(A)\geq
\Delta(A)$ for every set of vertices $A \subseteq V$.
\end{lemma}

Our next result relies on the following classical variant of the
"Max-flow Min-cut" theorem which deals with feasible flows in
Flow-Networks with upper and lower edge capacities (see e.g.
\cite{berge} p-88, where credit for that result is given to J.
Hoffman):

\begin{theorem} \label{uplow}
Given a flow network which consists of a  directed graph $D=(V,E)$
with upper and lower capacity functions $u:E \rightarrow R$ and
$l:E \rightarrow R$, $l(e) \leq u(e)$, there exists a {\bf
feasible} flow (no excess is allowed) $f:E\rightarrow R$ which
satisfies $l(e) \leq f(e) \leq
u(e)$ for every edge $e$, if and only if:\\
For every set of vertices $A\subseteq V$ \[\tag {*} \sum_{e\in
E^+(A)}l(e) \leq \sum_{e\in E^-(A)}u(e)\].
\end{theorem}

\subsection{An elemental theorem}

We can now state a necessary and sufficient condition for the
existences of an $(r,\alpha)$-flow
 in a cubic graph:

\begin{theorem} \label{first}
A cubic graph $G=(V,E)$ admits an $(r,\alpha)$-flow with $r\geq 2$
and $0 \leq \alpha < 3$ if and only if there exists a bisection
$(V_1,V_2)$ of $G$ which complies with the following two
conditions:
 \begin{enumerate}
\item For every subset $A$ of $V$, $d(A)\geq \Delta(A)$ and
\item For every subset $A$ of $V$, $\alpha \geq \frac{2d(A) -
(d(A)-\Delta(A))r}{2|A|}$
\end{enumerate}
Condition 1 becomes redundant if $(V_1,V_2)$ is known to be
orientable.
\end{theorem}

\begin{proof} Given an $(r,\alpha)$-flow in an orientation of a
cubic graph $G=(V,E)$, $\alpha < 3$ guaranties a balanced
orientation. Accordingly, $(V_1,V_2)$, where $V_i$, $i=1,2$ are
the sets of vertices with outdegree $i$, is an orientable
bisection and as such satisfies Condition 1.

Transform  now $G$ into a flow-network $\bar{G}$ by inserting one
additional "excess collector" vertex $x$ and a couple of
anti-parallel edges $xv$ and $vx$ for every vertex $v \in V$ of
$G$, one directed from $x$ toward $v$ and the other one from $v$
toward $x$. Set the lower capacity of the original edges of $G$ to
$1$ and their upper capacity to $r-1$. The edges incident with $x$
all get lower
capacity $0$ and upper capacity $\alpha$.\\
Equivalence between an $(r,\alpha)$-flow in $G$ and a feasible
flow in the network $\bar G$ is apparent.  $\bar G$ therefore
complies with condition (*) of Theorem \ref{uplow}. We now show
that Condition 2 (of Theorem \ref{first}) is satisfied by the
original graph $G$:

 When (*) is stated for a set $A
\subseteq V$, that is a set of vertices of $\bar G$ which does not
include $x$, it becomes:
\[\tag {**} 1 \cdot d^+_G(A)+0\cdot |A| \leq (r-1)d^-_G(A)+\alpha|A|\]
The first summand in each side reflects the capacity of the
original edges of $G$ and the second summand
relates to the edges incident with $x$.\\
(**) can be restated as:
\[\alpha \geq \frac{d^+_G(A)+d^-_G(A) - d^-_G(A)r}{|A|}\]

 Clearly $d^+(A)+d^-(A)=d(A)$,
$d^+(A)-d^-(A)=\delta(A)$, and hence $d^-(A)=\frac
{d(A)-\delta(A)}2$, which provides:
\[\tag {***} \alpha \geq \frac{2d(A) - (d(A)-\delta(A))r}{2|A|}\]

We should still consider sets of vertices of $\bar G$  which do
include the excess collector $x$. Such a set is of the form
$(V\setminus A) \cup \{x\}$ where $A$ is a subset of $V$. Notice
that the number of edges incident with $x$ in $E^+((V \setminus A)
\cup \{x\})$ and in $E^-((V \setminus A) \cup \{x\})$ is $|A|$.
Inequality (*) for that set is then:
\[\tag{****} 1 \cdot d^+(V\setminus
A)+0\cdot|A|\leq(r-1)d^-(V\setminus A)+\alpha|A|\] Since
$d^+(V\setminus A) =d^-(A)$ and $d^-(V\setminus A) =d^+(A)$,
(****) is obtained from (**) when $d^+$ and $d^-$ exchange roles,
which finally comes to replacing $\delta(A)$ by $-\delta(A)$ in
(***):
\[\tag{*****} \alpha \geq \frac{2d(A) - (d(A)+\delta(A))r}{2|A|}\]
Depending on the sign of $\delta(A)$, either (***) or (*****) is
redundant and they hence combine into Condition 2 of Theorem
\ref{first}.\\
For the proof of the "if" direction: Let $G=(V,E)$ and a bisection
$(V_1,V_2)$ of $G$ comply with Conditions 1 and 2.  By Lemma
\ref{orientable} and Condition 1 we assume an orientation of $G$
where $V_i$, $i=1,2$ are the sets of vertices with outdegree $i$.
Construct now the flow-network $\bar G$ as described above. Notice
that the derivation of Condition 2 from (*) of Theorem \ref{uplow}
is fully reversible: Condition 2 clearly implies (*****) and
(***). Then (**) is obtained from (***) by the substitutions
$d(A)=d^+(A)+d^-(A)$ and $\delta(A)=d^+(A)-d^-(A)$. Finally (**)
translates into (*) for all sets $A$ which do not include $x$.
Replacing (***) by (*****) takes care of all sets which do include
$x$. Theorem \ref{uplow} confirms the existence of a feasible flow
in $\bar G$ and equivalently an $(r,\alpha)$-flow in $G$.
\end{proof}

Theorem \ref{first} and its proof are generalizations of the case
$\alpha=0$, stated and proved in \cite{bvj} and in \cite{steff}.
When $\alpha=0$ Condition 1 is implied by Condition 2 and
therefore not explicitly stated in \cite{bvj} and in \cite{steff}.

\section{The bounded excess flows poset}
\subsection{Terminology}

 Whenever $r\leq s$ an
$r$-cnzf is {\bf stronger} than an $s$-cnzf, in the sense that the
existence of the first implies that of the second. In this section
we study the more complex two dimensional hierarchy among bounded
excess flows.

\begin{notation}
\leavevmode
\begin{itemize}
\item We use the notation $(r,\alpha)
\preceq(s,\beta)$
 when every cubic graph
which admits an $(r,\alpha)$-flow also admits an $(s,\beta)$-flow.
We say in that case that $(r,\alpha)$ is {\bf stronger} or
equivalent  to $(s,\beta)$.
\item $(r,\alpha)$ and
$(s,\beta)$ are {\bf equivalent} if $((r,\alpha) \preceq
(s,\beta)) \wedge ((s,\beta) \preceq (r,\alpha ))$.
\end{itemize}
\end{notation}
"Strong is Small" may be confusing, yet we prefer consistency with
nowhere zero flows.

Properties of the $\preceq$ order can be visualized in the
$r$-$\alpha$ plane (more accurately the upper right quadrant of
that plane, with $(2,0)$ as the origin). Some specific points,
lines and regions on that plane play major roles in our analysis.
To ease the formulation we adopt the following labeling
(Occasionally referring to Figure \ref{pos} while reading this and
the following sections is advised).

{\bf Remark}: Considering Theorem \ref{or5} the strength poset
collapses at $(3\frac12,\frac12)$ into a universal weakest
equivalence class, denoted in Figure \ref{pos} by $\Omega$.
Nonetheless, we chose to state our definitions, results and proofs
in  a general setting ("an integer $k$" rather than $k=4$ in some
cases, or $k<6$ in most). We found that approach preferable for
better insight into the subject without making things notably
harder or more complex than treating each case separately.
\begin{notation}
 \leavevmode
\item Given an integer $k\geq 3$, we define:
\begin{itemize}

\item $L_k$ is the segment of the line $\alpha=\frac{k-r}{k-2}$
between the points $(2,1$) (excluded) and $(k,0)$ (included).
\item $M_k$ is the upper part of $L_k$ from
$(3+\frac{k-3}{k-1},\frac{k-3}{k-1})$ (included) to $(2,1)$.
\item $A_k=\{(r,\alpha)|(2<
r<4)\wedge(\frac{k-r}{k-2}\leq\alpha<\frac{k-r+1}{k-1})\}$\\
is the half open triangle whose vertices are
$(3+\frac{k-3}{k-1},\frac{k-3}{k-1})$, $(2,1)$ and
$(4,\frac{k-3}{k-1})$, with the lower and the left $(M_k)$ edges
included, but not the upper-right edge and its two endvertices.

\item Let the {\bf Upper right domain}  $urd(r_0,\alpha_0)$ of a point
$(r_0,\alpha_0)$ be the upper-right closed unbounded polygonal
domain whose vertices are $(2,\infty),(2,1),(r_0,\alpha_0)$ and
$(\infty,\alpha_0)$.
\item For a cubic graph $G$, let the {\bf Bounded excess domain} $bed(G)$ be the set of all pairs
$(r,\alpha)$ such that $G$ admits an $(r,\alpha)$-flow.

\item If $D$ is a (balanced) orientation of $G$ then $bed(D)$ is the subset
of $bed(G)$ consisting of all points $(r,\alpha)$ for which there
exists an $(r,\alpha)$-flow in $D$.

\item For a given point $(r_0,\alpha_0)$ we define\\ {\bf
span}$(r_0,\alpha_0)=\{(r,\alpha)|(r_0,\alpha_0)\preceq
(r,\alpha)\}$. Equivalently,
\[span(r_0,\alpha_0)=\bigcap_{\{G|(r_0,\alpha_0)\in bed(G)\}} bed(G).\]
\end{itemize}

\end{notation}

\subsection{The Trace of a bounded excess flow and the role of $urd(r_0,\alpha_0)$}
When a balanced orientation $D$ (with the associated bisection
$\Psi(v)=d^+(v)$) of $G=(V,E)$ and a set $A\subseteq V$ (and
therefore $d(A),\Delta(A)$ and $|A|$) are kept fixed, Condition
$2$ of Theorem \ref{first}, as a linear inequality, corresponds to
an (upper right) half-plane of the $r$-$\alpha$ plane. Given a
balanced orientation $D$ of a cubic graph $G=(V,E)$, the set
$bed(D)$ of points $(r,\alpha)$ which complies with that
inequality for every $A\subseteq V$ is then an intersection of
(finitely many) half-planes and as such it is a $\bf convex$
unbounded polygonal domain. $bed(G)$ is the union of $bed(D)$ over
all balanced orientations $D$ of $G$. That is still an unbounded
polygonal domain. Convexity is not a-priori guaranteed, yet
clearly:

\begin{lemma}\label{conv}
If an $(r,\alpha)$-flow and an $(s,\beta)$-flow both exist in the
{\bf same orientation} $D$ of a cubic graph $G$, then the  line
segment between $(r,\alpha)$ and $(s,\beta)$ is entirely contained
in $bed(D)$, and therefore also in $bed(G)$.
\end{lemma}
In particular:
\begin{lemma} \label{21}
Every balanced orientation of every cubic graph $G$ admits a
$(2,1)$-flow.
\end{lemma}
\begin{proof} $f(e)=1$ for every edge $e$ clearly does the job.
\end{proof}

The points on the line defined by $(r_0,\alpha_0)$ and $(2,1)$ are
characterized by their:
\begin{definition}\label{trace}
The {\bf Trace} of a point $(r_0,\alpha_0)$, $r_0\geq 2,
\alpha_0<1$ is defined as:
\[tr(r_0,\alpha_0)=\frac{r_0-2\alpha_0}{1-\alpha_0}.\] We also define
$tr(f)=tr(r_0,\alpha_0)$ when $f$ is an $(r_0,\alpha_0)$-flow. The
line through $(2,1)$ and $(r_0,\alpha_0)$ intersects with the
$r$-axis at $r=tr(r_0,\alpha_0)$. In particular, the trace of an
$r$-cnzf is $r$.

Notice that the line segment $L_k$ consists of all the points $p$
for which $tr(p)=k$.
\end{definition}

Lemma \ref{conv} implies:

\begin{lemma}\label{second}
Let $p=(r_0,\alpha_0)$ and $q=(r_1,\alpha_1)$ be two points on the
$r$-$\alpha$ plane, such that $tr(p)=tr(q)$ and
$\alpha_0\leq\alpha_1$ (equivalently $r_0\geq r_1$)  then
$p\preceq q$.
\end{lemma}

 Obviously if $r \leq s$ and $\alpha \leq
\beta$ then $(r,\alpha)\preceq(s,\beta)$. Combining with Lemma
\ref{second}, it yields:

\begin{theorem}\label{third}
$urd(r,\alpha)\subseteq span(r,\alpha)$.
\end{theorem}

\subsection{Bounded excess flows and $k$-weak bisections} The
following definition is taken from \cite{emt2}:

\begin{definition}\label{weak}
Let $k\ge 3$ be an integer. A bisection $(V_1,V_2)$ of a cubic
graph $G=(V,E)$ is a \textbf{$k$-weak bisection} if every
connected component of each of the two subgraphs of $G$, induced
by $V_1$ and by $V_2$,  is a tree on at most $k-2$ vertices. Such
a component, as well as any subgraph with all vertices in the same
set $V_i$, is referred to in the sequel as \textbf{monochromatic}.
\end{definition}

A $k$-{\bf strong} bisection is also defined  in \cite{emt2} and
it is known \cite{bvj} to be equivalent to a $k$-nzf.

The following theorem reflects a similar connection between {\bf
orientable} $k$-weak bisections and bounded excess flows:

\begin{theorem} \label{main}
 Let $G$ be a cubic graph and $k\geq 3$ an integer.
The following three statements are equivalent:
\begin{enumerate}
\item $G$ admits an orientable $k$-weak bisection.
\item $M_k \subseteq bed(G)$.
\item $G$ admits a bounded excess flow $f$ with $tr(f)<k+1$.
\end{enumerate}
\end{theorem}

\begin{proof} $1\Rightarrow2$: Let $(V_1,V_2)$ be an orientable $k$-weak bisection of $G=(V,E)$ and let
$A\subseteq V$ be a set of vertices of $G$. Define $A_1=A\cap V_1$
and $A_2=A\cap V_2$. As $A_2$ induces a forest in a cubic graph,
$d(A_2)=|A_2|+2c$, where $c$ is the number of connected components
induced by $|A_2|$. Each such component has at most $k-2$ vertices
so $c\geq \frac {|A_2|}{k-2}$ and $d(A_2)\geq |A_2|+2\frac
{|A_2|}{k-2}= \frac k{k-2}|A_2|$. As $G$ is cubic at most $3|A_1|$
edges in $E(A_2)$ have their second endvertex in $A_1$.
Consequently
\[d(A) \geq \frac k{k-2}|A_2| - 3|A_1|\] We can assume $|A_2| \geq
|A_1|$ (otherwise exchange roles between $A_1$ and $A_2$), which
implies $\Delta(A)=\delta(A)$ so $|A_2|=\frac {|A|+\Delta(A)}2$
and $|A_1|=\frac {|A|-\Delta(A)}2$. When plugged into the last
inequality  it yields:

\[(k-2)d(A)+(k-3)|A|\geq(2k-3)\Delta(A)\]

We now divide by ${k-1}$ to get, after some manipulations:

\[\frac{k-3}{k-1}\geq \frac{2d(A) -(d(A)-\Delta(A))(3+\frac{k-3}{k-1})}{2|A|}\]

which is obtained from Condition 2 of Theorem \ref{first} with
$r=3+\frac{k-3}{k-1}$ and $\alpha=\frac{k-3}{k-1}$. Theorem
\ref{first} implies the existence of a
$(3+\frac{k-3}{k-1},\frac{k-3}{k-1})$-flow and by Lemma
\ref{second}, $M_k \subseteq bed(G)$.

 $2\Rightarrow3$: Points on $M_k$ are of trace $k<k+1$.

$3\Rightarrow1$: Assume an $(r_0,\alpha_0)$-flow in an orientation
of $G=(V,E)$. The partition $(V_1,V_2)$, where $V_i$, $i=1,2$ are
the sets of vertices with outdegree $i$, is clearly an orientable
bisection. If it is not a $k$-weak bisection then there exists a
monochromatic subgraph which is either a cycle on a set of
vertices $A$ where $d(A)=\Delta(A)=|A|$, or a tree on a set of
$k-1$ vertices $A$ where $d(A)=k+1$ and $|A|=\Delta(A)=k-1$.
Condition 2 of Theorem \ref{first} then either yields
$\alpha_0\geq 1$ which can be ignored, or

\[\alpha_0\geq \frac{2(k+1) - 2r_0}{2(k-1)}\] So $(r_0,\alpha_0)$ lies
on or on the right side of the line $\alpha=\frac{k+1-r}{k-1}$,
namely on or on the right side of  $L_{k+1}$, where the trace is
at least $k+1$. We proved $\neg 1\Rightarrow \neg 3$
\end{proof}

\subsection{Classification of some regions of the $r$-$\alpha$
plane} Properties of points in the labeled regions of the
$r$-$\alpha$ plane, depicted in Figure \ref{pos} can now be
deduce:
\begin{corollary}
No cubic graph admits an $(r,\alpha)$-flow $f$ with $tr(f)<3$
(that is $(r,\alpha)$ belongs to the region denoted by $O$ in
Figure \ref{pos})
\end{corollary}
\begin{proof}
Condition 2 of Theorem \ref{first}, when applied to a singleton
$A=\{x\}$, where $d(A)=3$ and $|A|=\Delta(A)=1$, yields $\alpha
\geq 3-r$, that is $tr(r,\alpha)\geq 3$ No $(r,\alpha)$-flow
therefore exists with $tr(r,\alpha)<3$.
\end{proof}
On the other extreme end:
\begin{lemma}
Every cubic graph $G=(V,E)$ admits a $(2,1)$-flow (no trace is
defined for that point).
 \end{lemma}
 \begin{proof} Considering Lemma \ref{21} it suffices to prove the existence of a balanced orientation:
 Turn $G$ into a $4$-regular graph $H$ by inserting $\frac{|V|}2$ new
edges which form a perfect matching. Select an Eulerian
orientation of $H$ and remove the extra edges.
\end{proof}

\begin{corollary} \label{ak}
Every two points in the same triangle $A_k$ are equivalent.

\end{corollary}
\begin{proof}
Let $p$ be a point in $A_k$. By Theorem \ref{third},
$(3+\frac{k-3}{k-1},\frac{k-3}{k-1})\preceq p$. Clearly
$tr(p)<k+1$ so by Theorem \ref{main}, $p \preceq
(3+\frac{k-3}{k-1},\frac{k-3}{k-1})$. All points of $A_k$ are then
equivalent to $(3+\frac{k-3}{k-1},\frac{k-3}{k-1})$ and
consequently also to each other.
\end{proof}

\begin{figure}[htb]
\centering \includegraphics[width=14cm]{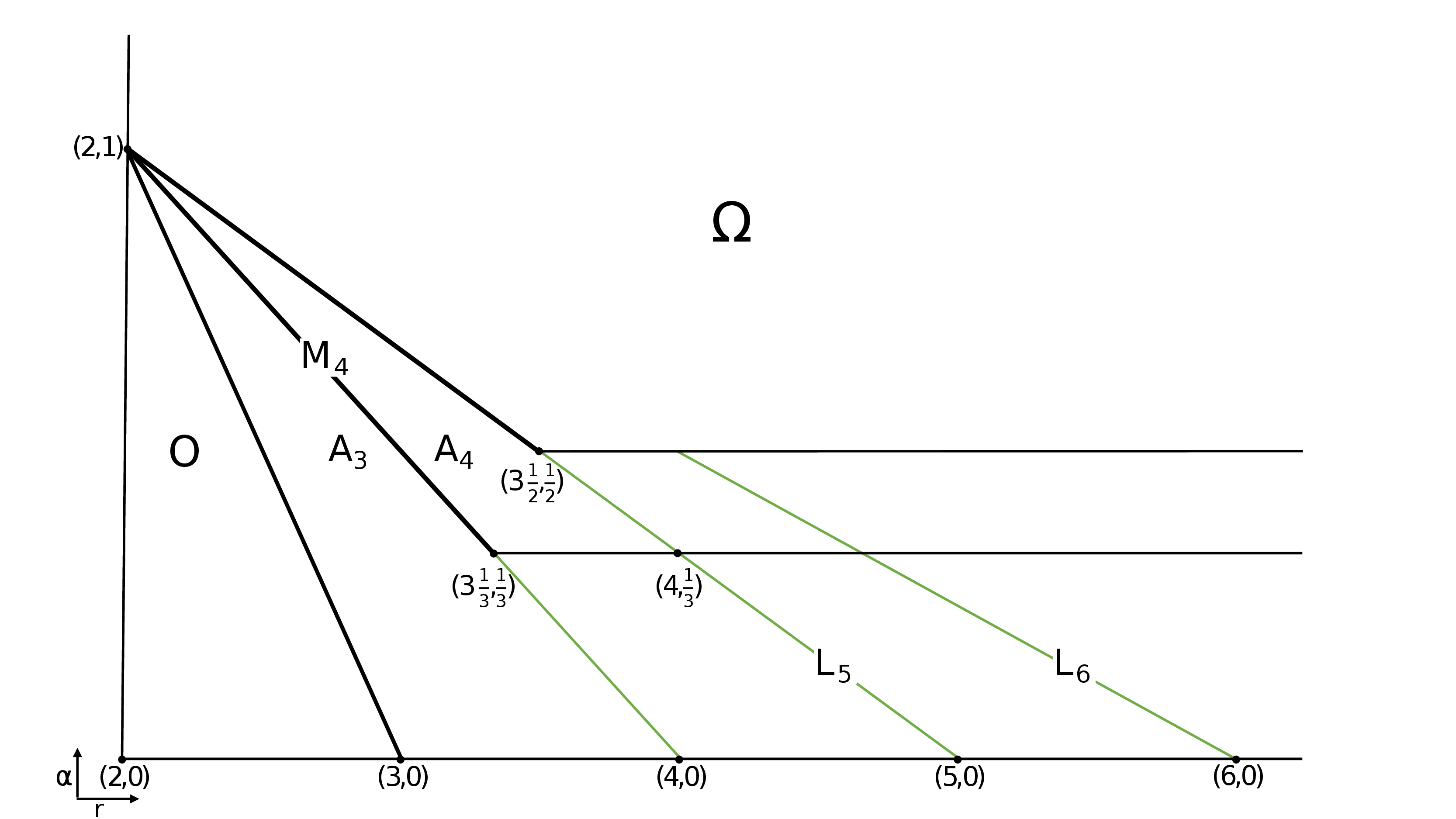} \caption{The
strength poset of $(r,\alpha)$-flows}\label{pos}
\end{figure}

Let us remark that every point $p$ such that $k\leq tr(p)<k+1$ is
at least as strong, but in general not equivalent to the points of
$A_k$. Furthermore:
\begin{theorem}\label{inf}
There are infinitely many non-equivalent points in the
$r$-$\alpha$ plane.
\end{theorem}
\begin{proof}
For every rational number $r_0$ in [4,5] There exists a cubic
graph $G$ with $\phi_c(G)=r_0$ \cite{lokut}. As a result, no two
points on the line segment from $(4,0)$ to $(5,0)$ (rational, or
not) are equivalent.
\end{proof}

Another characteristic of the $\preceq$ order is that it is a
proper partial order (not a full order).

\begin{theorem}
There are two points $p$ and $q$ such that neither $p\preceq q$
nor $q \preceq p$.
\end{theorem}
\begin{proof}
The Petersen graph admits a $5$-nzf, that is, a $(5,0)$-flow, but
it does not admit a $4$-weak bisection (see e.g. \cite{emt2}) and
therefore neither it admits a $(3\frac13,\frac13)$-flow (Theorem
\ref{main}). On the other hand take any cubic graph with a bridge
which does admit an orientable $4$-weak bisection and hence also a
$(3\frac13,\frac13)$-flow. Since the graph has a bridge it does
not admit any $r$-cnzf, in particular not a $5$-nzf. An example of
such a graph is presented by the diagram at the left side of
Figure \ref{butter2}. We believe (Conjecture \ref{bl3}) that any
cubic graph with a bridge which admits a perfect matching will do
as well.
\end{proof}

\subsection{Some noteworthy Instances and Corollaries} The
following is an explicit formulation of Lemma \ref{second} where
$p=(r_0,\alpha_0)$ are the points $(k,0)$, $k\in\{3,4,5,6\}$:

\begin{corollary} \label{3456}
\leavevmode
\begin{itemize}
\item A cubic graph is bipartite if and only if it admits a
$3$-nzf and therefore, an $(r,3-r)$-flow for every $r$, $2\leq r
\leq 3$.
\item
A cubic graph is $3$-edge-colorable if and only if it admits a
$4$-nzf and therefore, an $(r,\frac{4-r}{2})$-flow for every $r$,
$2\leq r \leq 4$.
\item
A cubic graph admits a $5$-nzf (every bridgeless cubic graph if
the assertion of the $5$-flow conjecture \cite{5flow} holds) if
and only if  it admits an $(r,\frac{5-r}{3})$-flow for every $r$,
$2\leq r \leq 5$.
\item
Every bridgeless cubic graph admits a $6$-nzf \cite{seym6} and
therefore an $(r,\frac{6-r}{4})$-flow for every $r$, $2\leq r \leq
6$.
\end{itemize}
\end{corollary}

 The following simple instance of Theorem \ref{main} applies to
 a significant family of Vizing's class two cubic graphs, which
 includes all "Classical Snarks" except for the Petersen graph (see e.g.
\cite{emt}).

\begin{corollary}
\label{k=4} If the circular flow number $\phi_c(G)$ of a
bridgeless cubic graph $G=(V,E)$ is strictly smaller than $5$ then
$G$ admits a $(3\frac 13,\frac 13)$-flow.
\end{corollary}
\begin{proof}
The trace of an $r$-cnzf is $r$. If $\phi_c(G)<5$ then $G$ admits
a flow $f$ with $tr(f)<5$ and Theorem \ref{main} applies.
\end{proof}

In the  following section we present the main result of this
article:

\section{Every cubic graph admits a $(3\frac12,\frac12)$-flow}\label{ikar}
\begin{theorem} \label{or5}
 Every cubic graph admits a $(3\frac12,\frac12)$-flow.
\end{theorem}

To prove Theorem \ref{or5} we have it restated by means of Theorem
\ref{main} as:

\begin{theorem} \label{t11}
Every cubic graph admits an orientable 5-weak bisection.
\end{theorem}

\begin{proof}
Theorem 11 of \cite{emt2} states the existence of a 5-weak
bisection for every cubic graph. The proof in \cite{emt2} however
does not guaranty an {\bf orientable} bisection. To reach that
goal we modify and significantly extend the proof of Theorem 11 of
\cite{emt2}.

The following lemma summarizes the part of our proof where we
explicitly rely on \cite{emt2}:

\begin{lemma} \label{factor}
A cubic graph $G=(V,E)$ contains a spanning factor $F$ where every
connected component is either a path on at lest two vertices, or a
cycle, and the following conditions hold:

\begin{enumerate}

\item the two endvertices of an odd (number of vertices) path in
$F$ ore non-adjacent.

\item An endvertex of one path in $F$ and an endvertex of
another path in $F$ are non-adjacent.

\item  A chord in an odd cycle of $F$ (if exits) is
parallel to an edge of that cycle.

\item If a vertex $y$ of an odd cycle $C$ of $F$ is adjacent to a vertex $x$ of another component
then $x$ is an {\bf internal} vertex of an odd path of $F$.

\item An odd path of $F$ is connected to at most one vertex of at
most one odd cycle of $F$.
\end{enumerate}
\end{lemma}

{\bf References} to these results and their proofs in Section 3 of
\cite{emt2} are (Claim numbers refer to that article. The factor
$F$ is denoted in \cite{emt2} by $P^*$): \leavevmode
\begin{enumerate}
\item Definition 12 (of \cite{emt2}).
\item Claim 13.
\item Claim 14.  In \cite{emt2} a graph is assumed to be {\bf
simple}. If parallel edges are allowed the proof of Claim 14
remains valid for the case where $u$ and $v$ are non-consecutive
vertices of $C$.

\item Claim 15. It is proved in \cite{emt2} that $x$ lies in an even position on an odd path,
but that additional fact is irrelevant to our needs here.
Furthermore, if $C$ is a chordless cycle in a simple cubic graph,
as it is assumed to be in \cite{emt2}, then
 {\bf every} vertex of $C$ is adjacent to a vertex out of
$C$. We have chosen a weaker formulation because this is not the
case when parallel edges are allowed  and we also wish the lemma
to apply to certain {\bf subgraphs} of $G$, where the degree of a
vertex can be less than $3$.
\item Claims 16 and 17
\end{enumerate}

\begin{notation} \label{ext}
\leavevmode \begin{itemize}
\item Until the end of the current section $F$ is a factor of a cubic
graph $G=(V,E)$, which complies with Lemma \ref{factor}. Every
connected component of $F$ is referred to as an {\bf
$F$-component}.
 \item A vertex $v \in V$ is {\bf
external} if it is an endvertex of a path of $F$, and also:
\item Two external vertices on each odd cycle $C$ are the two endvertices of
an arbitrarily selected {\bf simple} (not two parallel edges) edge
of $C$. At least every second edge along $C$ is simple so that
selection is clearly doable.
\item An edge which connects an external vertex of an odd cycle $C$
to a vertex which does not belong to $C$ (on an odd path by Lemma
\ref{factor}-4) is called a {\bf Critical} edge.
\item A vertex which is not external
is {\bf Internal}.
\item Let $C$ be an $F$-component on $k$ vertices
$v_1,v_2,...,v_k$ as ordered along $C$. Unless $C$ is an even
cycle, $v_1$ and $v_k$ are the external vertices. If $C$ is an
even cycle they are arbitrarily selected two consecutive vertices.
\item An alternating coloring $\Psi$ of $C$ is either a {\bf Parity}
coloring - $\Psi(v_i)=1$ if $i$ is odd and $\Psi(v_i)=2$ if $i$ is
even, or a {\bf Counter parity} coloring - $\Psi(v_i)=2$ if $i$ is
odd and $\Psi(v_i)=1$ if $i$ is even.
\item An alternating coloring of a subgraph of $G$ whose vertex
set is a union of $F$-components, is the union of alternating
colorings of these components.
\end{itemize}
 \end{notation}

 The following coloring rule will be obeyed in our construction of
 an orientable $5$-weak bisection of $G$:

 \begin{rrule}\label{crit}
If each of the two external vertices of an odd cycle $C$ is
incident with a critical edge then at least one of these two
critical edges is {\bf bi-chromatic} that is, its two endvertices
differ in color.
\end{rrule}

\begin{claim} \label{no4}
A monochromatic connected subgraph obtained by  an alternating
coloring which complies with Rule \ref{crit} is a tree on at most
three vertices.
\end{claim}
\begin{proof}
We say that a vertex is {\bf Good} if it differs in color from at
least two of its neighbors (a neighbor through two parallel edges
counts for that matter as two) and it is otherwise {\bf Bad}.
Observe that no vertex shares color with all its three neighbors
and hence a connected monochromatic subgraph is a tree on at most
three vertices if and only if it does not contain two adjacent bad
vertices of the same color. An internal vertex is clearly good. By
Lemma \ref{factor}-1,2,4 two external vertices of the same color
are adjacent only if they are the external vertices $u$ and $v$ of
an odd cycle $C$. We have to show that at least one of these two
external vertices is good. Let us assume that the common color of
$u$ and $v$ is 1. The other neighbor on $C$ of each one of them is
of color 2 (alternating coloring) so if one of them is adjacent to
that other neighbor through parallel edges then it is good.
Otherwise, Rule \ref{crit} guaranties for one of $u$ and $v$ a
second neighbor of color 2 which makes that endvertex good.
\end{proof}

Our goal is to describe an alternating coloring of $G$ which
complies with Rule \ref{crit} and also defines an orientable
bisection.

 Let us refer to an edge of $G$ which does not belong to $F$
 as a {\bf Skeletal} edge (s-edge). We now construct
a spanning graph $S$ of $G$ which fully contains all the
$F$-components and also a set $E_X$ of s-edges such that when
every $F$-component is contracted into a single vertex the graph
obtained from $S$ is a tree (For that to be possible we assume
that $G$ is connected).

The first step in the construction of $S$ is to include in $E_X$
every critical edge. This step is meant to gain control over the
edges relevant to Rule \ref{crit}. Lemma \ref{factor}-5 guaranties
no violation of the "tree like" structure of $S$.

To complete the construction of $S$ we add to $E_X$ additional
s-edges until the required property is reached, that is, $S$
becomes a tree when every $F$-component is contracted into a
single vertex. See Figure \ref{val}, where an $F$-component is
represented by a horizontal line (with an arc underneath if it is
a cycle) and the skeletal edges are vertical.

$S$ as a subgraph of $G$ is sub-cubic. We now generalize the
notion of an orientable bisection for graphs which are not
necessarily cubic. That definition is not associated with an
actual orientation but it carries the essence of Condition 1 of
Theorem \ref{first}:

\begin{definition} \label{orient}
A partition $\Psi=(V_1,V_2)$ of an even (number of vertices) graph
$H=(V',E')$ is an orientable bisection if for every set of
vertices $A\subseteq V'$
\[d_H(A)\geq\Delta_{\psi}(A).\]
The inequality above where $A=V'$ implies that $\Psi$ is indeed a
bisection.
\end{definition}

\begin{notation}\label{skel}
\leavevmode
\begin{itemize}
\item In the sequel we construct a coloring $\Psi$ of $S$ which complies
with Rule \ref{crit} and with the inequalities of Definition
\ref{orient}. Let's call such a coloring $\Psi$ a {\bf Valid}
coloring. An example of the obtained coloring is depicted in
Figure \ref{val}.

\item The removal of $k$ s-edges decomposes $S$ into $k+1$ connected
subgraphs. Each of these components is referred to as a {\bf
skeletal} subgraph (s-subgraph). An s-subgraph either entirely
contains an $F$-component $C$, or it is disjoint from $C$.
 \item An s-subgraph $H=(V',E')$ is even or odd according to the parity
of $|V'|$.
 \item We say that an s-edge $e$ of $S$ is even or odd
according to the parity of each of the two s-subgraphs obtained by
the removal of $e$ (the same parity since $S$ is even). Similarly
$e$ is even or odd in an even s-subgraph $H$ of $S$ according to
the parity of the two subgraphs of $H$ obtained by the removal of
$e$.
\end{itemize}
\end{notation}

\begin{claim} \label{par}
The removal of a sequence of  even s-edges decomposes $S$ into
{\bf even} disjoint s-subgraphs.
\end{claim}
\begin{proof}
Assume to the contrary a minimal sequence $Q$ of even s-edges,
such that the removal of the last (the order does not really
matter) one $e$ decomposes one of the  existing even components
into two odd ones. At that stage the decomposition includes
exactly two odd components and the others are all even. Reinsert
the edges of $Q \setminus \{e\}$ one by one. On each such step two
components are merged into one. As $e$ remains outside, the two
odd components remain separated from each other, so after each
step there are still exactly two odd components. Finally when $e$
is the only edge left outside, $S$ is decomposed into two odd
s-subgraph, in contradiction with the assumption that $e$ is even.
\end{proof}

It is worth noting  that even s-subgraphs can also be generated by
the removal of  {\bf odd} s-edges.  e.g. the removal of two
s-edges of which one is odd provides one  even s-subgraph $H$ (and
two odd ones), such that an even edge of $S$ may be odd in $H$.
The assertion of the last claim, therefore, cannot be taken for
granted without a proof.

\begin{definition}
A {\bf Prime even} s-subgraph (pes-subgraph) is obtained from $S$
by repeatedly removing even s-edges until the remaining s-edges
are all odd.
\end{definition}.

\begin{claim} \label{noeven}
In order to prove the existence of a valid coloring of $S$ as well
as of every even s-subgraph $H$ obtained from $S$  by the removal
of any set of even s-edges, it suffices to prove the existence of
such a coloring for every pes-subgraph. When applying Rule
\ref{crit} to a subgraph $H$ we should consider only  odd cycles
with two critical edges which belong both to $H$.
 \end{claim}

      \begin{proof}
       IF $H$ is a pes-subgraph then we are done. Otherwise, let $e$ be an even s-edge whose removal decomposes
       $H$ into two even subgraphs $H_1$ and $H_2$. We can assume by
       induction the existence of valid colorings of $H_1$ and of $H_2$.
        The union $\Psi$ of these two
       colorings is clearly an orientable alternating bisection of $H$.  As for
       Rule \ref{crit}, by induction it is satisfied for odd cycles with two critical edges in the
        same subgraph $H_i$. Attention should be paid to the
       case where the removed edge $e$ is a critical edge of an odd cycle $C$, in
       one of the subgraphs. Notice that switching
       between the colors 1 and 2 in one of the two {\bf even}
       subgraphs $H_1$ or $H_2$
       does not compromise the  orientable bisection. That way
       we can guaranty the two endvertices of $e$ to be of distinct colors, as
       required by Rule \ref{crit}. See Figure \ref{val}
\end{proof}

\begin{notation}
 \leavevmode
\begin{itemize}
\item Let $H$ be a pes-subgraph of $S$. We select an $F$-component $R_H$ to be the {\bf root} of $H$.
\item The removal of an s-edge $e$ decomposes $H$ into two odd
 s-subgraphs, the {\bf Lower
 side} of $e$,  $D(e)$ which contains $R_H$ and its {\bf Upper side} $U(e)$ which does not
contain $R_H$.
\item Accordingly, the endvertices of $e$ are its {\bf Upper} endvertex and its {\bf Lower} endvertex.
\item The subgraph $\bar e$ consists of $e$ and its two endvertices.
\item The union $B$ of $\bar e$ and $U(e)$ is a {\bf Branch} of $H$.
$\bar e$ is the {\bf Stem} of the branch $B=\bar e \cup U(e)$ and
$U(e)$ is the {\bf Top} of $B$, also denoted by $t(B)$.
\item The $F$-factor $C$ in $t(B)$ which includes the upper vertex
of the stem is the {\bf Base} of the branch $B$.
\item The stem of a branch $B$ is also refereed to as the stem of the base of
$B$. Every $F$-factor in $H$, other than the root $R_H$ is the
base of a branch and hence has a stem.
 \item The lower endvertex of the stem $\bar e$ of a branch $B$ is
 the {\bf Heel of $B$}, denoted by $heel(B)$, while the upper endvertex is the {\bf
 Heel of $\bar e$}. The reason for that, somewhat confusing
 terminology will be cleared soon.
 \item A branch can also be obtained by contracting the lower side
 of its stem into a single vertex (the heel). Contracting an odd
 subgraph into a single vertex preserves subgraph parity so, like
 $H$, every branch is prime-even in the sense that the removal of
 any s-edge provides two disjoint {\bf odd} subgraphs. Accordingly we
 define a prime-even {\bf generalized} s-subgraph (pegs-subgraph)
 $T$ of a pes-subgraph $H$ to be either $H$ itself or a branch in
 $H$.
 \item The base of a pes-subgraph $H$ when referred to as a
 pegs-subgraph is its root $R_H$.
 \item Let $T$ be a pegs-subgraph and let $C$ be the base of $T$.
 A branch $L$ whose heel belongs to $C$ is a {\bf limb} of $T$.
 When $T$ is a branch, the stem $\bar e$ of $T$ also counts as one of its
 limbs.
 \item The upper endvertex  of the stem $\bar e$ of a branch belongs to the
 base of that branch. For that reason, $heel(\bar e)$ is
 counterintuitively its upper endvertex as previously defined. The lower endvertex of the stem
 forms its top subgraph $t(\bar e)$.
 See Figure \ref{branch}, where the stem is drawn from the base
 $C$ upwards with its lower endvertex at the top.
In Figure \ref{val} we get a more global view where the stems are
drawn from
 each components downwards toward the root.
 \end{itemize}
\end{notation}

Obviously:
\begin{observation}
All clauses of Lemma \ref{factor} apply to a pes-subgraph $H$ (as
well as to any s-subgraph) of $S$, where $F$ is restricted to the
$F$-components which are contained in $H$.
\end{observation}

As for a branch $B$, Lemma \ref{factor}-1,2,3,5 still similarly
apply. Lemma \ref{factor}-4 however, poses an issue when the base
of $B$ is an odd cycle $C$ and the heel of the stem is incident
with an external vertex of $C$. In that case the "out of $C$"
neighbor is the lower endvertex of the stem which is a single
vertex while the $F$-component to which it belongs is not
contained in $B$.

We say that an odd cycle $C$ in a pegs-subgraph $T$ is {\bf
bi-critical} in $T$ if both its external vertices are incident
with critical edges which belong to $T$.

Let us restate a stronger version of Rule \ref{crit} which applies
to any pegs-subgraph $T$, be it a branch  or a pes-subgraph:

\begin{rrule} \label{crit1} Let $C$ be a bi-critical odd cycle in a
pegs-subgraph $T$. At list one critical edge of $C$ which {\bf
does not belong to the stem} of $C$ should be bi-chromatic, that
is, its two endvertices should differ in color.
\end{rrule}
\begin{definition} A {\bf valid orientable bisection} of a pegs-subgraph $T=(V',E')$
 is an alternating coloring $\Psi$ of  $T$ which complies with Rule
\ref{crit1} and with the condition $d(A)_T\geq\Delta_{\Psi}(A)$
for every subset of vertices $A \in V'$.
\end{definition}

We have set the necessary tools to state and prove:

\begin{lemma} \label{nstep}
 Every pegs-subgraph $T$ admits a valid orientable bisection.
 \end{lemma}

 \begin{proof}

Let  the base $C$ of $T$ be an $F$-component on $k$ vertices
 $v_1,v_2,...,v_k$ as ordered along $C$. As previously stated, if $C$ is not an even
 cycle then $v_1$ and $v_k$ are selected to be the external vertices of $C$. If
 $C$ is an even cycle then $v_1$ and $v_k$ are two arbitrarily
 selected consecutive vertices along $C$.

 Notice that a
 vertex of $C$ is not necessarily incident with an s-edge and if $C$ is a path each of $v_1$ and $v_k$
  may be incident with two s-edges, accordingly, the
 number $m$ of limbs of $T$ can be smaller or
 larger or equal to $k$. Nonetheless:

 \begin{claim} \label{parity}
 The number $m$ of limbs of $T$
 is of the same parity as the order $k$ of the base $C$.
\end{claim}
\begin{proof}
The number of limbs $m$ is also the number of skeletal edges
incident with $C$.
 Removal of these s-edges decomposes $T$ into $m+1$ components of which $m$ are odd limbs' tops
and  the last one is the base $C$. The parity of the total number
of vertices in $T$ is
 therefore the parity of $m+k$. As $T$ is even that implies $m\equiv k(modulo~2)$
 (See Figure \ref{val}. Remember to count the stem among the limbs
 of each branch).
 \end{proof}

Let $L_1,L_2,...,L_m$ be the limbs of $T$ in nondecreasing order
of the indices of their heels among $v_1,v_2,...,v_k$.  The proof
proceeds by induction: We assume  the existence of a valid
bisection $\Psi_j$ for every limb $L_j$ of $T$ (verification for
the smallest pegs-subgraphs is left for the end of the proof).

Let $K_j$ be the vertex set of $t(L_j)$ the top of $L_j$ ($K_j$
includes all vertices of $L_j$ except for its heel).

 As $\Psi_j$ is a bisection and the heel of $L_j$ is a
single vertex, $\Delta_{\psi_j}(K_j)=1$. We say that the {\bf
color of the top set $K_j$} is 2 if $\delta_{\psi_j}(K_j)=1$, or
it is 1 if $\delta_{\psi_j}(K_j)=-1$.

Observe that $\Psi_j$ remains valid if the colors 1 and 2 are
switched. Accordingly, we can freely chose the color of each limb
top $K_j$ without violating the validity of the colorings
$\Psi_j$.

We perform that choice according to the following rules:
\begin{rrule} \label{psij}
\leavevmode
\begin{itemize}
\item If the base $C$ is any $F$-component other than a bi-critical odd cycle then
 $K_j$ gets its {\bf counter-parity color}, that is 2 if $j$ is odd and 1 if $j$ is
 even.

\item If $C$ is a bi-critical odd cycle in $T$ then assume that $v_1$ is adjacent to a vertex
$x$ on an odd path in $K_1$ and that the limb $L_1$ is not the
stem of $C$ (otherwise reverse the order of $v_1,...v_k$). For
$2\leq j\leq m-1$ let $K_j$ get its parity color (1 if $j$ is odd
and 2 if $J$ is even). Now select a color for $K_1$ such that
$\Psi_1(x)=2$. Conclude with coloring $K_m$ to make its color
distinct from the color of $K_1$ (See Figure \ref{branch}).

\item So far we have an orientable  bisections $\Psi_j$ for each
limb. Their union  however does not necessarily covers the entire
subgraph $T$, as some vertices of $C$ may not belong to limbs.
Also the sought coloring should be alternating on each
$F$-component.
 Thus we finalize the definition of a bisection $\Psi$ of $T$ by recoloring the vertices
 $v_1,...,v_k$ each with its parity color. Clearly an alternating
 coloring.
 \end{itemize}
\end{rrule}

\begin{figure}[h!]
\centering \includegraphics[width=12cm]{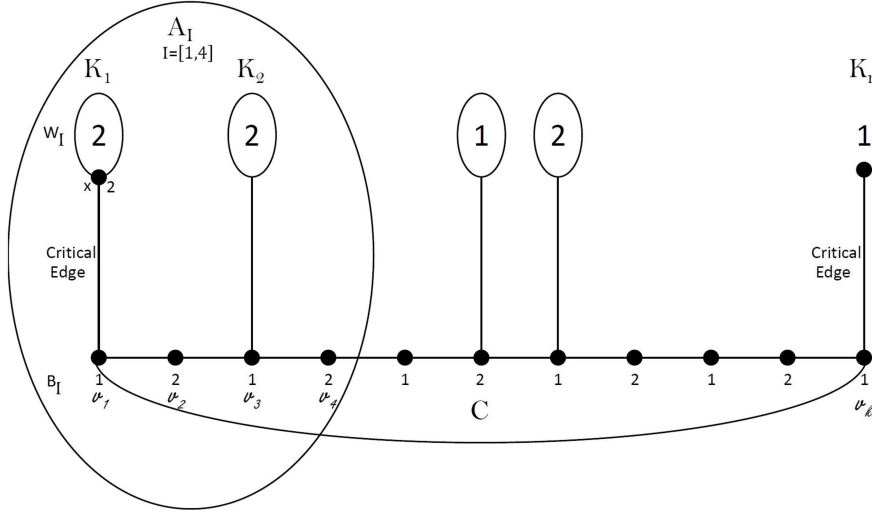}
\caption{Coloring of a branch with a bi-critical odd cycle as a
base}\label{branch}
\end{figure}

 Following Rule \ref{psij}, one can verify that in either
 case half of the $k+m$ elements - vertices of $C$ and
 limbs' tops - are colored 1 and the other half are colored 2.
Consequently,  $\Psi$ is indeed a bisection. In making that
observation notice the equal parity
 of $k$ and $m$ and the fact that when both are odd the color 1 has majority in $C$
  and the color $2$ gains
 majority among the limbs' tops (As before, the color $\Psi(K_j)$ of a top set is the one with majority
 among the vertices of $K_j$. The contribution of a top
  set to $\delta_{\Psi}$ is the same as that
 of a single vertex of the same color).

 Let us now verify that $\Psi$ is a valid orientable
 bisection of $T$:

Rule \ref{crit} in its more specific formulation as Rule
\ref{crit1}, can be assumed (induction) to be obeyed by each of
the colorings $\Psi_j$ of the limbs $L_j$. Inserting the base $C$
of $T$ may be relevant to the rule only if $C$ is either a
bi-critical odd cycle in $T$ or an odd path (Lemma
\ref{factor}-4).

Assume that $C$ is  a bi-critical odd  cycle. We chose in that
case the color of $K_1$ such that the vertex $x$ to which the
external vertex $v_1$ is adjacent, is of color 2. Then $v_1$ as a
vertex of $C$ is assigned with its parity color, namely 1 (Rule
\ref{psij}). Consequently, Rule \ref{crit1} is obeyed also by the
new coloring $\Psi$ of $T$ (See Figure \ref{branch}, where $K_1$
is assumed to be colored 2 for the endvertex $x$ of the critical
edge to get the color 2).

We now consider the case where the base  $M_j$ of a certain limb
$L_j$ is a bi-critical odd cycle, and the base $C$ of $T$ is an
odd path. $M_j$ in that case is connected to $C$ by its stem,
which may include an external vertex of $M_j$. Nonetheless, Rule
\ref{crit1} when applied to $L_j$ (induction) guaranties that an
external vertex $y$ of $M_j$, which does not belong to the stem of
$M_j$ is already adjacent through a critical edge to a vertex $x$
which differs from $y$ in color. So we are good in that case as
well.

The vertices of each $F$-component got their parity color so
$\Psi$ is an alternating coloring as required for a valid
bisection.

It remains to show that $\Psi$ is an orientable bisection.

\begin{claim} \label{suff}
For the condition $d(A)\geq\Delta(A)$ of Definition \ref{orient}
it suffices to consider sets of vertices $A$ such that for every
limb $L_j$ of $T$, $A$ includes either all or none of the vertices
of $L_j$.
\end{claim}
\begin{proof}
Let $A$ be a set of vertices of $T$ which does not include
$heel(L_j)$ but does contain a non empty subset $A'$ of $K_j$. it
implies that $A \setminus A'$ is disjoint form $L_j$ so we can
assume to have proved \[d(A \setminus A')\geq\Delta_{\Psi}(A
\setminus A').\] $\Psi$ and $\Psi_j$ are identical on $K_j$ and
therefore on $A'$. So $\Delta_{\Psi}(A')=\Delta_{\Psi_j}(A')$
(Regardless of the color $\Psi(heel(L_j))$ which might have
changed). As $\Psi_j$ is an orientable bisection \[d(A')\geq
\Delta_{\Psi}(A')\] the relevant edge cuts are disjoint so
\[d(A)=d(A\setminus A')+d(A')\] Also
$\delta_{\Psi}(A)=\delta_{\Psi}(A\setminus A')+\delta_{\Psi}(A')$
namely
\[\Delta_{\Psi}(A)\leq\Delta_{\Psi}(A\setminus
A')+\Delta_{\Psi}(A')\]

It all comes to the required inequality \[d(A)\geq
\Delta_{\Psi}(A)\] $\Psi$ is a bisection and as such provides the
same values of $d$ and $\Delta$ to a set and to its complement.
Accordingly, the above also applies to sets $A$ which do include
$heel(L_j)$.
\end{proof}

In addition to complete limbs of $T$ the set $A$ may include some
vertices of $C$ which are not incident with s-edges.

For an integer $i, 1\leq i \leq k$ let $P_i$ be the set of
vertices of the limb whose heel is $v_i$, or the singleton
$\{v_i\}$ if there is no such limb.

Accordingly, a relevant set $A_I$, is defined by a set of indices
$I \in \{1,2,...,k\}$  as:
\[A_I= \bigcup_{i\in I} P_i\]

Vertices of $C$ which do not belong to $I$ separate $I$ into a set
$\cal J$ of disjoint intervals of consecutive integers. It is
apparent that $d$ as well as $\delta$ can be computed separately
for each interval in $\cal J$ and then sum up to obtain  $d(A_I)$
and $\delta(A_I)$. Consequently:

\begin{claim}
For the proof of  $d(A)\geq\Delta(A)$ It suffices to consider sets
$A_I$ where I is an interval $[i_l,i_r]=\{i|i_l\leq i \leq i_r\}$
of consecutive integers between 1 and $k$.
\end{claim}

$d(A_I)=2$ whenever $I$ is {\bf Internal}, that is, if $C$ is a
cycle or $C$ is a path where $1\notin I$ and $k \notin I$. If $C$
is a path then $d(A_I)=1$ if $I$ is a {\bf Terminal Interval}
which either include 1 or $k$ (if both then $A_I$ is the entire
set of vertices with $d= \delta=0$).

Let $I$ be the interval $[i_l,i_r]$. For the computation of
$\delta(A_I)$ we represent the colors of the vertices of $A_I$ by
two sequences of the colors 1 and 2:
\begin{itemize}
\item The {\bf Base sequence} $B_I$ consisting of the colors
$\Psi(v_{i_l}),\Psi(v_{{i_l}+1}),...,\Psi(v_{i_r})$. Since $\Psi$
is an alternating coloring this is always an alternating sequence.
\item The second is the {\bf Top sequence} $W_I$ of the colors
 $\Psi(K_{j_l}),...\Psi(K_{j_r})$ of the top sets of the limbs which belong to $A_I$.
 These colors are set according to Rule \ref{psij}.
\end{itemize}

Let us consider first the case where the base $C$ of $T$ is not a
bi-critical odd cycle. In that case both, the base and the top
sequences are alternating (See Rule \ref{psij}). The difference
between the numbers of 2s and 1s in an alternating sequence is at
most 1, so it sums up to at most 2 for the union of the two
sequences. That comes to $\Delta(A_I)\leq 2$. As $I$ is internal
$d(A_I)=2$ so
\[d(A_I)\geq \Delta(A_I)\]
as required.

If $I$ is a terminal interval of a path, say $I=[1,i_r]$ then
either $\delta(B_I)=0$ if $|I|$ is even, or $\delta(B_I)=-1$ if
$|I|$ is odd. Rule \ref{psij} implies in that case that the first
term of $W_I$ is $\Psi_1(K_1)=2$ (unless $W_I$ is empty, which
makes no exception), so either $\delta(W_I)=0$ or $\delta(W_I)=1$.
Summing this up provides $\Delta(A_I)\leq 1$. As $I$ is terminal
$d(A_I)=1$ so the required inequality holds. If $I=[i_l,k]$ we
rely on the equal parity of $k$ and $m$ so this time the last
(rather than the first) terms of $B_I$ and of $W_I$ are distinct,
which leads to the same computation and final result.

Now to the case where the base $C$ is a bi-critical odd cycle in
$T$. In that case $v_1$ and $v_k$ both are colored 1 and both are
heels of limbs of $T$. Rule \ref{psij} guaranties the top of one
of these two limbs to be colored 2. This limb, say $L_1$ (with no
loss of generality) contributes a 1 to the base sequence and a 2
to the top sequence. When $L_1$ is removed both the base and the
top sequences (for the entire subgraph $T$) become alternating
even sequences, where $\Delta(A_I)\leq 2$ for every interval $I$.
That does not change when $1\in I$ because $\delta(P_1)=0$. Since
$C$ is a cycle there are no terminal intervals and $d(A_I)=2$ for
every interval $I$. $d(A_I)\geq \Delta(A_I)$ follows (See Figure
\ref{branch}).

To initialize the induction, Lemma \ref{nstep} should be verified
where $T$ consists of a single $F$-component. $T$ can either be an
even isolated $F$-component, or a branch which consists of an odd
base $C$ and its stem.  An isolated even $F$-component provides an
even alternating base sequence and an empty top sequence. An odd
base $C$ with a stem yields an odd base sequence and a top
sequence consisting of a single 2. Neither of the two poses an
exception to the proof schema described above.
\end{proof}

\begin{figure}[htb]
\centering \includegraphics[width=14cm]{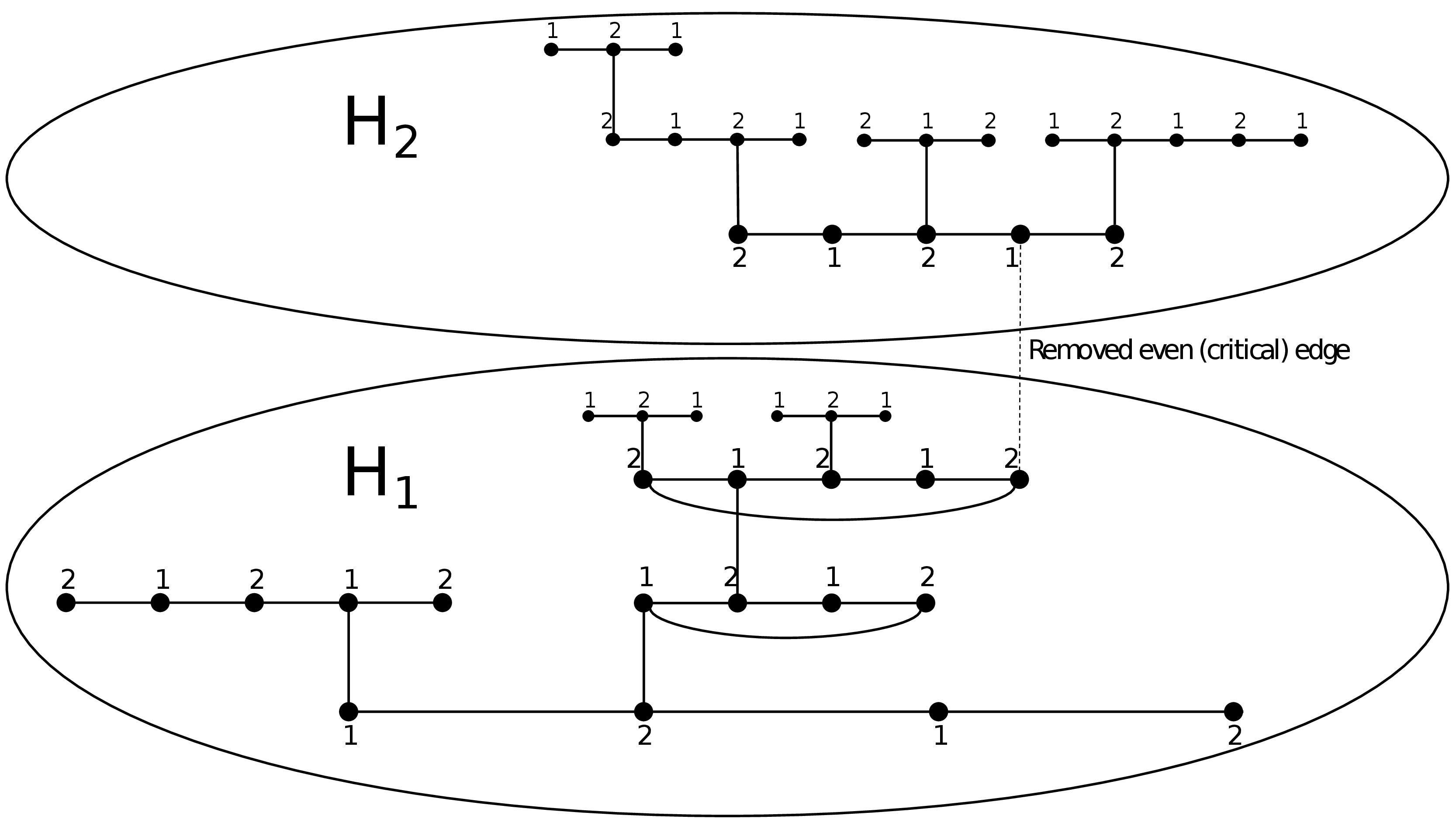}
\caption{Explicit valid coloring of the union of two adjacent
prime even s-subgraphs}\label{val}
\end{figure}

{\bf Concluding the proof of Theorem \ref{t11} (and Theorem
\ref{or5}):}

 Lemma \ref{nstep} where $T=S$ provides an alternating orientable bisection $\Psi$ of $S$.

Rule \ref{crit} is meant to guaranty that at least one of the two
external vertices of an odd cycle is "good" in the sense that it
differs in colors from at least two of its neighbors. Let us
summarize the verification of that property for the original cubic
graph $G$:

In the proof of Lemma \ref{nstep} we took care of odd cycles with
two critical edges within a pegs-subgraph $T$. Let's observe that
it indeed suffices:

When constructing $S$ we started with including in $S$ every
critical edge of $G$. If an  external vertex $y$ of an odd cycle
$C$ is not incident with a critical edge within a pegs-subgraph
$T$ then either there is no such edge in $G$ or it was removed as
an even edge when decomposing $S$ into pes-subgraphs. In the first
case $y$ connects to a neighbor $x$ on $C$ with two parallel
edges. That makes $y$ good because $x$ counts as two and its color
differs from the color of $y$ (alternating coloring). In the
second case $y$ can be considered good as shown in the proof of
Claim \ref{noeven}.

As $S$ is a spanning subgraph of $G$ Both graphs $G$ and $S$ share
the same vertex set $V$. Therefore the bisection $\Psi$ applies to
$G$ as well as it does for $S$ with the same value of
$\Delta_{\Psi}(A)$ for every $A\subseteq V$. On the other hand,
the edge set of $S$ is a subset of $E$ which implies $d_G(A)\geq
d_S(A)$. The inequalities $d(A)\geq \Delta(A)$ therefore hold for
the bisection $\Psi$ of $G$ and $\Psi$ is indeed an alternating
orientable bisection of $G$, which complies with Rule \ref{crit}.

Claim \ref{no4} then asserts that a connected monochromatic
subgraph induced by $\Psi$ is a tree on at most 3 vertices. By
Definition \ref{weak}, $\Psi$ is an orientable 5-weak bisection of
$G$.
\end{proof}

Theorem \ref{or5} cannot be improved as it provides a tight
result:

\begin{theorem} \label{mesh}
There exists a cubic graph $G$ with $bed(G)=urd(3\frac12,\frac12)$
\end{theorem}
\begin{proof}

\begin{figure}[h!]
\centering \includegraphics[width=12cm]{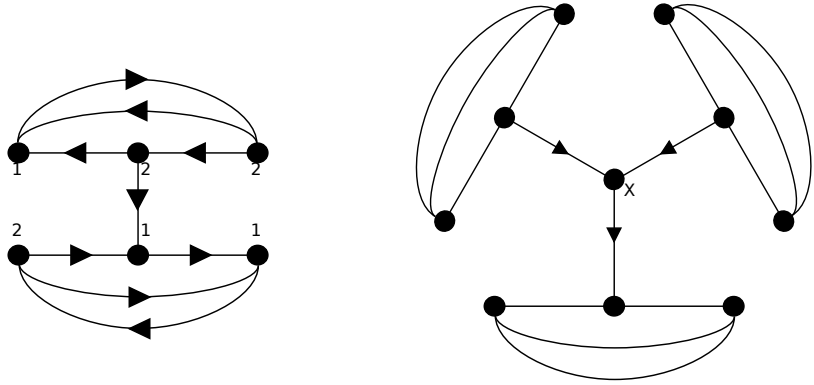} \caption{An
orientable $4$-weak bisection with no $5$-nzf (left)
$~~~~~~~~~~~~~~~~~~~~~~~~~$ A graph $G$ with
$bed(G)=urd(3\frac12,\frac12)$ (right)}\label{butter2}
\end{figure}

Let $f$ be an $(r,\alpha)$-flow in any balanced orientation of the
graph $G$ on the diagram at the right side of Figure
\ref{butter2}. Two of the three edges incident with the vertex $x$
are directed both into $x$, or both from $x$ outwards. The total
flow on these two edges is at least $2$ and it spreads as
(positive or negative) excess among four vertices, which makes
$\alpha \geq \frac12$. A brief case analysis shows that $G$ does
not admit a $4$-weak bisection (any bisection either induces two
monochromatic parallel edges or a monochromatic path on three
vertices). By Theorem \ref{main} $tr(f)\geq 5$ and since $\alpha
\geq \frac12$ it implies $(r,\alpha)\in urd(3\frac 12,\frac12)$,
so $bed(G)\subseteq urd(3\frac 12,\frac12)$ . Theorem \ref{or5} on
the other hand yields $(3\frac 12,\frac 12)\in bed(G)$.
$bed(G)=urd(3\frac 12,\frac12)$ follows.
 \end{proof}

 A straightforward conclusion is:
 \begin{corollary}
 The {\bf intersection of $bed(G)$ over all cubic graphs} $G$ is
 $span(3\frac12,\frac12)=urd(3\frac12,\frac12)$. That region is denoted by $\Omega$
 in Figure \ref{pos}.
\end{corollary}

{\bf Remark:} If restricted to cubic graphs which admit a perfect
matching (bridgeless graphs included), Theorem \ref{or5} has a
much simpler and ten times shorter proof. It is basically the
second proof of Theorem 11, presented for that restricted case in
\cite{emt2}, with some modifications.

\section{Open problems and concluding remarks}
\subsection{Two Conjectures}
$\phi_c(G)<5$ is not a necessary condition for the result in
Corollary \ref{k=4}. Following A. Ban and N. Linial \cite{banlin},
the {\bf revised Ban Linial Conjecture} (Conjecture 9 in
\cite{emt2}) asserts that every  cubic graph which admits a
perfect matching, other than the Petersen graph, admits a $4$-weak
bisection. We hereby suggest the following stronger version:

\begin{conjecture} \label{bl3} Every  cubic graph $G$ which admits a perfect matching, other
than the Petersen graph, admits an {\bf orientable} $4$-weak
bisection and equivalently, a $(3\frac 13,\frac 13)$-flow.
\end{conjecture}

A perfect matching is not necessary for a $4$-weak bisection.
However, infinitely many cubic graphs which admit no $4$-weak
bisection are presented in \cite{emt2}, so the scope of Conjecture
\ref{bl3} cannot be extended further to include all cubic graphs.

 The graph in Figure \ref{butter2} contains pairs of parallel edges.
 When analyzing the smallest example, presented in \cite{emt2}, of a
{\bf simple} cubic graph $G$ with no $4$-weak bisection we found,
along similar lines, that $bed(G)=urd(4\frac14,\frac14)$ (notice
that $(4\frac14,\frac14)$, of trace $5$ lies on the lower part of
$L_5$). We have reasons to believe:
\begin{conjecture}
For every simple cubic graph $G$, $(4\frac14,\frac14)\in bed(G)$.
\end{conjecture}

 \subsection{What does $bed(G)$ look like?}
 Little do we actually know about the shape of $bed(G)$ in
 general.

 For every cubic graph $G$, $bed(G)$ is a closed unbounded (to the
 upper-right) polygonal domain. Always among its sides are two
 infinite ones, a vertical side on the line
 $r=2$ from $(2,\infty)$ to $(2,1)$ and a horizontal one on $\alpha=\alpha_m$ from a
 certain point $(r_m,\alpha_m)$ to $(\infty,\alpha_m)$. If $G$ is
 bridgeless then $(r_m,\alpha_m)=(\phi_c(G),0)$. If there exists
 a bridge in $G$ then a lower bound for $\alpha$ is $\frac
 1{|V_m|}$, where $|V_m|$ is the number of vertices in the smaller
 side of the bridge. The actual minimum value $\alpha_m$ may be
 larger than that, see Theorem \ref{mesh} and its proof. By
 theorem \ref{or5} the point $(3\frac12,\frac12)$ is always in
 $bed(G)$.

 Results in this article almost solely rely on the analysis of
 a certain single (balanced) orientation, rather than understanding the union of
 $bed(D)$ over several orientations $D$ of a graph $G$.
 The proof of Theorem \ref{main} asserts that the existence of a bounded excess flow $f$
 with $tr(f)<k+1$ implies the existence of a
 $(3+\frac{k-3}{k-1},\frac{k-3}{k-1})$-flow in the same orientation $D$ as
 $f$. By Lemma \ref{conv} the line segment between the
 corresponding two points is contained in $bed(D)$ and in $bed(G)$.
 Other than the above, what we can add at that stage, are mostly
 questions. Following is a rather arbitrary list of questions. At that point we cannot tell how
 hard or easy they are and how interesting
 the answers may be:
\begin{enumerate}
\item Take a bridgeless graph $G$ with $4<\phi_c(G)<5$, say
$\phi_c(G)=4 \frac12$  (For existence see e.g. \cite{lokut}). The
 line segment from $(3\frac13,\frac13)$ to $(4\frac12,0)$ is
contained in $bed(G)$. Is it a side of $bed(G)$?
\item Are the vertices of $bed(G)$ of the previous question
$(2,\infty),(2,1),(3\frac13,\frac13),(4\frac12,0)$ and
$(\infty,0)$?
\item if $G$ is bridgeless, does $bed(G)$  depends solely on
$\phi_c(G)$?
\item Does every cubic graph $G$ have a {\bf Dominant orientation}
$D$ such that $bed(G)=bed(D)$, that is, $bed(D')\subseteq bed(D)$
for every orientation $D'$ of $G$?
 \item Is $bed(G)$ always convex (It sure is if the answer to the previous question is affirmative)?

 \item Is there a constant bound to the number of sides of
 $bed(G)$ (that is a bound to the number of sets $A$ relevant to
 Condition 2 of Theorem \ref{first})?

\item Considering the proof of Theorem \ref{inf}: Are there two equivalent
points in the quadrilateral whose four vertices  are
$(4,0),(5,0),(4,\frac13)$ and $(3\frac13,\frac13)$?

\item Given a cubic graph $G$, does there always exist a point
$p$ in the $r$-$\alpha$ plane such that $bed(G)=span(p)$? If
exists, such a point represent a strongest bounded excess flow in
$G$, which is a two-dimensional generalization of the
circular-flow number $\phi_c(G)$.

\item Inspired by the 5-flow Conjecture: Is there a cubic graph
$G$ such that $bed(G)$ has a finite vertex whose trace is larger
than 5 (obviously true if the assertion of the 5-flow Conjecture
is false)?

\item In the quest for settling the 5-flow Conjecture, can we prove
the existence of a point $(r_0,\alpha_0)$ with
$tr(r_0,\alpha_0)=5$ and $\alpha_0<\frac12$ (equivalently
$r_0>3\frac12$)
 such that every {\bf bridgeless} cubic graph
admits an $(r_0,\alpha_0)$-flow?

\end{enumerate}


\begin{thebibliography}{99}

\bibitem{banlin} A. Ban and N. Linial, \emph{Internal Partitions of
  Regular Graphs},  J. Graph Theory {\bf 83(1)} (2016), 5--18.

\bibitem{berge} C. Berge, \emph{Graphs and Hypergraphs}, North Holland (1973)

\bibitem{emt} L. Esperet, G. Mazzuoccolo and M. Tarsi
\emph{The structure of graphs with circular flow number 5 or more,
and the complexity of their recognition problem} J. Comb. {\bf
7(2) (2016)}, 453--479.


\bibitem{emt2} L. Esperet, G. Mazzuoccolo and M. Tarsi
\emph{Flows and Bisections in Cubic Graphs} J. Graph Theory {\bf
86(2) (2017)}, 149--158.


\bibitem{bvj} F. Jaeger, \emph{Balanced Valuations and Flows in
Multigraphs}, Proc. Amer. Math Soc. {\bf 55(1)} (1976), 237--242.

\bibitem{lokut} R. Lukot'ka and M. Skoviera, \emph{Snarks with given real flow
numbers}, J. Graph Theory {\bf 68(3)} (2011), 189--201.

\bibitem{seym6} P.D. Seymour,  \emph{Nowhere-zero 6-flows}, J. Combin.
Theory Ser. B  {\bf 30(2)} (1981), 130--135.

\bibitem{steff} E. Steffen, \emph{Circular flow numbers of regular multigraphs}, J. Graph
Theory {\bf 36(1)} (2001), 24--34.

\bibitem{5flow} W.T. Tutte, \emph{A contribution to the theory of chromatic
polynomials}, Canad. J. Math. {\bf 6} (1954), 80--91.

\bibitem{cq} C Q. Zhang, \emph{Integer Flows and Cycle Covers of Graphs}, New York Marcel Dekker
 (1997).

\end{thebibliography}
\end{document}